\documentclass[11pt]{amsart}

\usepackage{amssymb,amsmath,latexsym,epsfig,euscript}

\theoremstyle{plain}
\newtheorem{theorem*}{Theorem}[]{}
\newtheorem{theorem}{Theorem}[section]
\newtheorem{proposition}[theorem]{Proposition}

\theoremstyle{definition}
\newtheorem{definition}[theorem]{Definition}

\theoremstyle{remark}
\newtheorem{remark}[theorem]{Remark}

\newtheorem{example}[theorem]{Example}

\newcommand{\nc}{\newcommand}

\nc{\fg}{\mathfrak{f} } \nc{\vg}{\mathfrak{v} } \nc{\wg}{\mathfrak{w} }
\nc{\zg}{\mathfrak{z} } \nc{\ngo}{\mathfrak{n} } \nc{\kg}{\mathfrak{k} }
\nc{\mg}{\mathfrak{m} } \nc{\bg}{\mathfrak{b} } \nc{\ggo}{\mathfrak{g} }
\nc{\ggob}{\overline{\mathfrak{g}} } \nc{\sog}{\mathfrak{so} }
\nc{\sug}{\mathfrak{su} } \nc{\spg}{\mathfrak{sp} } \nc{\slg}{\mathfrak{sl} }
\nc{\glg}{\mathfrak{gl} } \nc{\cg}{\mathfrak{c} } \nc{\rg}{\mathfrak{r} }
\nc{\hg}{\mathfrak{h} } \nc{\tg}{\mathfrak{t} } \nc{\ug}{\mathfrak{u} }
\nc{\dg}{\mathfrak{d} } \nc{\ag}{\mathfrak{a} } \nc{\pg}{\mathfrak{p} }
\nc{\sg}{\mathfrak{s} } \nc{\pca}{\mathcal{P}} \nc{\nca}{\mathcal{N}}
\nc{\lca}{\mathcal{L}} \nc{\oca}{\mathcal{O}} \nc{\mca}{\mathcal{M}}
\nc{\tca}{\mathcal{T}} \nc{\aca}{\mathcal{A}} \nc{\cca}{\mathcal{C}}
\nc{\gca}{\mathcal{G}} \nc{\sca}{\mathcal{S}} \nc{\hca}{\mathcal{H}}
\nc{\bca}{\mathcal{B}} \nc{\dca}{\mathcal{D}} \nc{\val}{\operatorname{val}}

\nc{\vp}{\varphi} \nc{\ddt}{\tfrac{{\rm d}}{{\rm d}t}}
\nc{\dpar}{\tfrac{\partial}{\partial t}} \nc{\im}{\mathtt{i}}
\renewcommand{\Re}{{\rm Re}}

\nc{\SO}{\mathrm{SO}} \nc{\Spe}{\mathrm{Sp}} \nc{\Sl}{\mathrm{SL}}
\nc{\SU}{\mathrm{SU}} \nc{\Or}{\mathrm{O}} \nc{\U}{\mathrm{U}} \nc{\Gl}{\mathrm{GL}}
\nc{\Se}{\mathrm{S}} \nc{\Cl}{\mathrm{Cl}} \nc{\Spein}{\mathrm{Spin}}
\nc{\Pin}{\mathrm{Pin}} \nc{\G}{\mathrm{GL}_n(\RR)} \nc{\g}{\mathfrak{gl}_n(\RR)}

\nc{\RR}{{\Bbb R}} \nc{\HH}{{\Bbb H}} \nc{\CC}{{\Bbb C}} \nc{\ZZ}{{\Bbb Z}}
\nc{\FF}{{\Bbb F}} \nc{\NN}{{\Bbb N}} \nc{\QQ}{{\Bbb Q}} \nc{\PP}{{\Bbb P}}

\nc{\vs}{\vspace{.2cm}} \nc{\vsp}{\vspace{1cm}} \nc{\ip}{\langle\cdot,\cdot\rangle}
\nc{\ipp}{(\cdot,\cdot)} \nc{\la}{\langle} \nc{\ra}{\rangle} \nc{\unm}{\tfrac{1}{2}}
\nc{\unc}{\tfrac{1}{4}} \nc{\und}{\tfrac{1}{16}} \nc{\no}{\vs\noindent}
\nc{\lam}{\Lambda^2(\RR^n)^*\otimes\RR^n} \nc{\tangz}{{\rm T}^{\rm Zar}}
\nc{\nor}{{\sf n}}  \nc{\mum}{/\!\!/} \nc{\kir}{/\!\!/\!\!/}
\nc{\Ri}{\tfrac{4\Ric_{\mu}}{||\mu||^2}} \nc{\ds}{\displaystyle}
\nc{\ben}{\begin{enumerate}} \nc{\een}{\end{enumerate}} \nc{\f}{\frac}
\nc{\lb}{[\cdot,\cdot]} \nc{\isn}{\tfrac{1}{||v||^2}}

\nc{\He}{\operatorname{Hess}} \nc{\ad}{\operatorname{ad}}
\nc{\Ad}{\operatorname{Ad}} \nc{\rank}{\operatorname{rank}}
\nc{\Irr}{\operatorname{Irr}} \nc{\End}{\operatorname{End}}
\nc{\Aut}{\operatorname{Aut}} \nc{\Inn}{\operatorname{Inn}}
\nc{\Der}{\operatorname{Der}} \nc{\Ker}{\operatorname{Ker}}
\nc{\Iso}{\operatorname{I}} \nc{\Diff}{\operatorname{D}} \nc{\Lie}{\operatorname{L}}
\nc{\tr}{\operatorname{tr}} \nc{\dif}{\operatorname{d}}
\nc{\sen}{\operatorname{sen}} \nc{\modu}{\operatorname{mod}}
\nc{\Ric}{\operatorname{R}} \nc{\Ricci}{\operatorname{Ric}}
\nc{\sym}{\operatorname{sym}} \nc{\symac}{\operatorname{sym^{ac}}}
\nc{\symc}{\operatorname{sym^{c}}} \nc{\scalar}{\operatorname{sc}}
\nc{\grad}{\operatorname{grad}} \nc{\ricci}{\operatorname{ric}}
\nc{\ricciac}{\operatorname{ric^{ac}}} \nc{\riccic}{\operatorname{ric^{c}}}
\nc{\riccig}{\operatorname{ric^{\gamma}}} \nc{\Rin}{\operatorname{M}}
\nc{\Le}{\operatorname{L}} \nc{\tang}{\operatorname{T}}
\nc{\level}{\operatorname{level}} \nc{\rad}{\operatorname{r}}
\nc{\abel}{\operatorname{ab}} \nc{\CH}{\operatorname{CH}}
\nc{\mcc}{\operatorname{mcc}} \nc{\Adj}{\operatorname{Adj}}
\newcommand{\bi}{\overline{i}}

\newcommand{\bk}{\overline{k}}
\newcommand{\bl}{\overline{l}}
\newcommand{\bm}{\overline{m}}

\newcommand{\br}{\overline{r}}

\begin{document}

\title{Quasi-K\"ahler Chern-flat manifolds and complex $2$-step nilpotent Lie algebras}

\author{Antonio J. Di Scala} \author{Jorge Lauret} \author{Luigi Vezzoni}

\address{Dipartimento di Matematica, Politecnico di Torino, Torino, Italy.}
\email{antonio.discala@polito.it}

\address{FaMAF and CIEM, Universidad Nacional de C\'ordoba, C\'ordoba, Argentina.}
\email{lauret@famaf.unc.edu.ar}

\address{Dipartimento di Matematica, Universit\`a di Torino, Torino, Italy.}
\email{luigi.vezzoni@unito.it}

\thanks{The first and the third authors were supported by the Project M.I.U.R. ``Riemannian Metrics and  Differentiable Manifolds''
and by G.N.S.A.G.A. of I.N.d.A.M.\\
The second author was supported by a grant from SeCyT (Universidad Nacional de C\'ordoba)}

\date{}

\maketitle

\begin{abstract}
The study of quasi-K\"ahler Chern-flat almost Hermitian manifolds is strictly related to the study of anti-bi-invariant almost complex Lie algebras.
In the present paper we show that quasi-K\"ahler Chern-flat almost Hermitian structures on compact manifolds are in correspondence to
complex parallelisable Hermitian structures satisfying the second Gray identity.
From an algebraic point of view this correspondence reads as a natural correspondence between
anti-bi-invariant almost complex structures on Lie algebras to bi-invariant complex structures.
Some natural algebraic problems are approached and some exotic examples are carefully described.
\end{abstract}

\section{Introduction}
A complex manifold is called \emph{complex parallelisable} if
admits a global holomorphic $(1,0)$-frame. In view of a result of Wang
\cite{wang}, a compact complex manifold is complex parallelisable if and only if it is
a quotient of a complex Lie group by a discrete subgroup. Moreover, a complex manifolds is complex parallelisable if and
only if admits a \emph{Chern-flat} metric, i.e. a metric whose Chern connection has trivial holonomy.
The aim of this paper is to study the interplay between compact quotients of compact Lie groups and
quasi-K\"ahler Chern-flat manifolds from the point of view of both Lie theory and
Hermitian geometry.
Quasi-K\"ahler Chern-flat manifolds were introduced in \cite{DV}.
A compact almost complex manifold is called quasi-K\"ahler Chern-flat if admits a compatible quasi-K\"ahler metric (also called $(1,2)$-symplectic) such that its Chern connection has trivial holonomy group. In view of a result proved in \cite{DV}, compact quasi-K\"ahler Chern-flat manifolds are all obtained as quotients of $2$-step nilpotent Lie groups by a lattice. From an algebraic point of view, the study of complex Lie groups corresponds to the study of bi-invariant complex structures on Lie algebras, while  the study of quasi-K\"ahler Chern-flat manifolds corresponds to the study of $2$-step nilpotent almost complex Lie algebras $(\mathfrak{g},\lb,J)$ satisfying the \emph{anti-bi-invariant} relation $J\lb=-[J\cdot,\cdot]$.
A $J$ satisfying the previous relation will be called in the paper an \emph{anti-bi-invariant }structure.
These algebraic correspondences allows us to relate the geometry of compact quotient of complex Lie groups to the geometry of quasi-K\"ahler Chern-flat manifolds.

More precisely, let $G$ be a complex Lie group with complex structure $J$, Lie algebra $(\mathfrak{g},\lb)$ and a left-invariant Hermitian metric $g$.  Let $\Gamma \subset G$ be a lattice and let $M=G/\Gamma$.  Then $J$ induces the anti-bi-invariant almost complex structure
$$
J^{-}(X)=\begin{cases}
+J(X)\quad \mbox{if }X\in \mathfrak{z} \\
-J(X)\quad \mbox{if }X\in \mathfrak{z}^{\perp}\,,
\end{cases}
$$
where $\mathfrak{z}$ is the center of $\mathfrak{g}$. We will prove the following

\begin{theorem}\label{th1}
The following facts are equivalent:
\begin{enumerate}
\item[1.] the curvature of the metric $g$ satisfies the second Gray identity $;$
\vspace{0.1cm}
\item[2.] the curvature of any left-invariant Hermitian metric on $M=G/\Gamma$ satisfies the second Gray identity $({\rm G_2})$ $;$
\vspace{0.1cm}
\item[3.] $g$ is quasi-K\"ahler with respect to $J^-;$
\vspace{0.1cm}
\item[4.] the Lie group $G$ is $2$-step nilpotent$\,.$
\end{enumerate}
Moreover, if one of the above conditions holds, then $g$ is Chern-flat with respect to $J$ and $J^-$.
\end{theorem}
In other words, the correspondence $J\mapsto J^-$ preservers the \emph{Chern flat} condition, while requiring that $J^-$ is also quasi-K\"ahler is equivalent to require that $R^g$ satisfies the second Gray identity  with respect to $J$. Gray identities are classical identities involving the curvature of an almost Hermitian structure and they will be recall in the
next section.

Theorem \ref{th1} will be used to study the interplay between some spaces of almost complex structures on compact manifolds.
Indeed, given a compact Riemannian manifold $(M,g)$, we consider the following spaces:
\begin{itemize}
\item $\mathcal{Z}_g=\Big\{\mbox{almost complex structures on $M$ compatible to $g$}\Big\}\,;$

\vspace{0.1cm}
\item $\mathcal{Q}\mathcal{K}^{0}_{g}=\Big\{J\in\mathcal{Z}_g\,:\mbox{ $(g,J)$ is a quasi-K\"ahler Chern-flat almost Hermitian structure} \Big\}$\,;

\vspace{0.1cm}
\item $\mathcal{CP}^{2}_{g}=\Big\{ J\in\mathcal{Z}_g\,:\mbox{ $(g,J)$ is a Chern-flat Hermitian structure and $R^g$ satisfies $({\rm G}_2)$} \Big\}$\,.
\end{itemize}
Then we will prove the following
\begin{theorem}\label{th2}
There exists a natural one-to-one correspondence between $\mathcal{Q}\mathcal{K}^{0}_{g}$ and $\mathcal{CP}^{2}_{g}$.
Moreover, there exists an {\rm SKT} structure $J$ in $\mathcal{CP}^{2}_{g}$ if and only if
$(M,g,J)$ is a flat K\"ahler torus.
\end{theorem}
According to the results of \cite{DV} described above, compact quasi-K\"ahler
Chern-flat manifolds are all $2$-step nilmanifolds whose Lie algebras are endowed
with an anti-bi-invariant structure.  Two anti-bi-invariant structures are considered {\it isomorphic} if they are
conjugate by an automorphism of $\ngo$. The following natural questions arise:
\begin{itemize}
\item[(Q1)] Which real $2$-step algebras admit an anti-bi-invariant structure?  Is a `reasonable' classification of anti-bi-invariant structures up to isomorphism feasible?

\item[(Q2)] Is any anti-bi-invariant $2$-step algebra $(\ngo,J)$ the anticomplexification of a real $2$-step algebra?

\item[(Q3)] Given an anti-bi-invariant $2$-step algebra $(\ngo,J)$, are there any canonical or distinguished almost Hermitian metrics on $(\ngo,J)$?

\item[(Q4)] Can a given anti-bi-invariant $2$-step algebra $(\ngo,J)$ be the
anticomplexification of infinitely many pairwise non-isomorphic real $2$-step
algebras?

\item[(Q5)] Can a given $2$-step algebra $\ngo$ admit two non-isomorphic anti-bi-invariant
structures?
\end{itemize}

We approach these `algebraic' questions in Section \ref{alg}.  Let $\ngo$ be a real
$2$-step algebra of dimension $2n$.  It is easy to see that an almost complex
structure $J$ on $\ngo$ is anti-bi-invariant if and only the almost
complex structure $J^-$ obtained by changing the sign of $J$ on the
$J$-invariant complement subspace of $[\ngo,\ngo]$, is {\it bi-invariant} (i.e.
$[JX,Y]=J[X,Y]$ for all $X,Y\in\ngo$).  This answers (Q1) as follows: $\ngo$ admits
an anti-bi-invariant structure if and only if $\ngo$ is a complex $n$-dimensional
$2$-step algebra viewed as real.  Moreover, the classification of anti-bi-invariant
$2$-step algebras up to isomorphism is equivalent to the classification of complex
$2$-step algebras up to isomorphism, a wild problem (see Section \ref{varphi}).

If $(\hg_{\CC},J)$ is the anticomplexification of a real $n$-dimensional $2$-step
algebra $\hg$, then $J^-$ is precisely the multiplication by $\im=\sqrt{-1}$ in $\hg_{\CC}$.
This implies that an anti-bi-invariant $(\ngo,J)$ is the anticomplexification of
$\hg$ if and only if $(\ngo,J^-)$ is the complexification of $\hg$, and so any
complex $2$-step algebra without any real form provides a negative answer to (Q2)
(see Section \ref{comp}).

Concerning (Q3), let us first recall that the precise meaning of canonical or
distinguished is part of the question.  An almost Hermitian metric on an almost
complex Lie algebra is said to be {\it minimal} if it minimizes the norm of the
$(1,1)$-component of the Ricci tensor among all almost Hermitian metrics with the
same scalar curvature.  We prove that an anti-bi-invariant almost Hermitian
$(\ngo,J,\ip)$ is minimal if and only if $(\ngo,\ip)$ is a {\it Ricci soliton}, i.e.
the solution of the Ricci flow $\dpar g(t)=-2\,{\rm Ric}_{g(t)}$ starting at $\ip$
evolves only by scaling and the action by diffeomorphisms (see Section \ref{can}).

Question (Q4) has a negative answer, provided by the well-known fact that a complex Lie algebra admits at most finitely many real forms
up to isomorphism (see Section \ref{comp}).

Question (Q5) remains open. It is equivalent to the existence of two
non-isomorphic complex $2$-step algebras which are isomorphic as real Lie algebras.

\vs \noindent {\it Acknowledgements.}  We are grateful to Michael Jablonski for providing us with the reference \cite{BrlHrs}, which was used in Remark \ref{MJ}. We are also grateful to the referee for useful suggestions and remarks.

\section{Preliminaries}\label{pre}
An \emph{almost Hermitian manifold} is an even dimensional Riemannian manifold $(M,g)$
equipped with an almost complex structure $J$ which is an isometry with respect to
$g$.  Any almost Hermitian manifold admits a canonical Hermitian connection. More precisely:
\begin{proposition}
Let $(M,g,J)$ be an almost Hermitian manifold. There exists a unique Hermitian
connection $\nabla$ on $M$ $($i.e. $\nabla J=\nabla g=0)$ whose torsion has
everywhere vanishing $(1,1)$-part.
\end{proposition}
This connection was introduced by Ehresmann and Libermann in \cite{EL}. Since in the
complex case $\nabla$ coincides with the connection used by Chern in \cite{chern},
it is called the \emph{Chern connection} (it is often also called the \emph{second
canonical Hermitian connection} or simply the \emph{canonical connection}).
\begin{definition}
An \emph{almost Hermitian metric} $g$ is called \emph{Chern flat} if
the holonomy group of the associated Chern connection is trivial.
\end{definition}

On a compact complex manifold $(M,J)$ the existence of a Chern flat metric is equivalent to require that
that $(M,J)$ is complex parallelisable.
We recall that a complex manifold $(M,J)$ is called
\emph{complex parallelizable} if the complex bundle $T^{1,0}M$ is trivial as
holomorphic bundle. This condition is equivalent to require that there exists a
global $(1,0)$-frame $\{Z_1,\dots,Z_n\}$ on $M$ such that the mixed brackets
$[Z_i,Z_{\br}]$ vanish. We recall the following
\begin{theorem}[Wang \cite{wang}]\label{wangth}
Let $(M,J)$ be a compact complex parallelisable manifold. Then there exists a simply-connected complex
Lie group $G$ and a discrete subgroup $\Gamma$ of $G$ such that $(M,J)= G/\Gamma$.
Moreover, the holomorphic cohomology ring of $M$ is isomorphic with the holomorphic
cohomology ring of the Lie algebra of $G$. Finally, $M$ cannot be simply-connected
and it is K\"ahlerian if and only if it is a complex torus.
\end{theorem}
Moreover we have the following
\begin{proposition}\label{wang2}
Let $(M,J)$ be a compact complex manifold. Then $(M,J)$ is complex parallelisable if and only if
there exists a  Chern-flat Hermitian metric on $(M,J)$. Moreover, if $(M=G/\Gamma,J)$ is a compact complex parallelizable manifold, then a Hermitian metric $g$ on $M$ is Chern-flat
if and only if it is left-invariant.
\end{proposition}

Note that if $(M= G/\Gamma,J)$ is a compact complex parallelisable manifold, then $J$ induces a
bi-invariant complex structure on the Lie algebra $\mathfrak{g}$ of  $G$. \emph{Bi-invariant} means that $[JX,Y]=J[X,Y]$ for every
$X,Y\in\mathfrak{g}$. Then if $\mathfrak{g}$ is the Lie algebra of $G$ we have $[\mathfrak{g}^{1,0},\mathfrak{g}^{0,1}] = 0$.

Now we recall some facts about the geometry of Chern flat metric on almost complex manifold.
Any almost Hermitian manifold $(M,g,J)$ has a natural non-degenerate $2$-form  $\omega$ defined by the relation $\omega(\cdot,\cdot)=g(\cdot,J\cdot)$. An
almost Hermitian manifold $(M,g,J)$ is called \emph{quasi-K\"ahler} (or
$(1,2)$-symplectic) if $\overline{\partial}\omega=0$, where $\overline{\partial}$ is
computed with respect to $J$. The quasi-K\"ahler condition is equivalent to require
that the vector field $D_{\overline{Z}_1}Z_2$ is of type $(1,0)$ for any pair
$(Z_1,Z_2)$ of vector fields of type $(1,0)$, where $D$ denotes the Levi-Civita
connection of $g$. We have the following
\begin{theorem}[\cite{DV}, Theorem 1]\label{DVth1}
Let $(M,J)$ be a compact almost complex manifold. The following facts are
equivalent:
\begin{enumerate}
\item[1.] $(M,J)$ admits a quasi-K\"ahler Chern flat metric$;$
\vspace{0.1cm}
\item[2.] there exists a global $(1,0)$-coframe $\{\zeta_1,\dots,\zeta_n\}$ on $M$ such that $\partial\zeta_r=\overline{\partial}\zeta_r=0\,;$
\vspace{0.1cm}
\item[3.] $(M,J)$ is a $2$-step nilmanifold equipped with a left-invariant almost complex structure such that
\begin{equation}\label{J[}
[JX,Y]=-J[X,Y]
\end{equation}
for any pair $(X,Y)$ of left-invariant vector fields.
%Moreover, any almost Hermitian
%left-invariant metric $g$ on $(M,J)$ is automatically quasi-K\"ahler and Chern-flat.
\end{enumerate}
Moreover, if one of the above condition holds , then a Hermitian metric $g$ on $M$ is quasi-K\"ahler Chern flat
if and only if it is left-invariant.
\end{theorem}
This motivates the following
\begin{definition}[\cite{DV}] An almost complex manifold $(M,J)$ is called \emph{quasi-K\"ahler Chern flat}, if there exists a quasi-K\"ahler Chern flat metric $g$ on $M$ compatible to $J$.
\end{definition}
Theorem \ref{DVth1} motives the study of almost complex Lie algebras equipped with an
almost complex structure satisfying condition \eqref{J[}. We can give the following
\begin{definition}
An almost complex structure $J$ on a Lie algebra $(\mathfrak{g},\lb)$ is \emph{anti-bi-invariant} if $[JX,Y]=-J[X,Y]$ for any $X,Y\in\mathfrak{g}$.
\end{definition}
%In the sequel will call
%such almost complex structures \emph{anti-bi-Hermitian}.

%\medskip
%
%Now we give an explicit formula for the curvature and the Ricci tensor of a quasi-K\"ahler Chern-flat metric:\\
%\noindent Let $g$ be a quasi-K\"ahler Chern flat metric on a compact almost complex manifold   $(M^{2n},J)$. There
%exists a global unitary frame $\{Z_1,\dots,Z_n\}$ such that
%\begin{enumerate}
%\item[i.] $[Z_i,Z_{\bj}]=0$;
%
%\vspace{0.1cm}
%\item[ii.] $[Z_i,Z_j]=\sum_{r=1}^n B_{ij}^{\br}Z_{\br}$ and the $B_{ij}^{\br}$'s are constant on $M$.
%\end{enumerate}
%Let $D$ be the Levi-Civita connection of $g$. We can write
%$$
%D_{i}Z_{j}=\sum_{r=1}^n\Gamma_{ij}^{\br}Z_{\br}\,,
%$$
%where the $\Gamma_{ij}^{\br}$'s are related with the components of the Lie brackets by
%the Koszul formula
%\begin{equation}\label{Gamma}
%\Gamma_{ij}^{\br}=\frac12\left(B_{ij}^{\br}+B_{ri}^{\bj}-B_{jr}^{\bi}\right)\,.
%\end{equation}
%Properties i and ii allows us to describe the curvature and the Ricci tensor of the
%metric $g$:
%\begin{proposition}\label{Ricci}
%The curvature of $\nabla$ is determined by the following formulae
%\begin{equation}\label{R}
%R_{i\bj k\bl}=-\sum_{r=1}^n\Gamma_{ik}^{\br}\Gamma_{\bj\br}^{l}\,,\quad
%R_{ij\bk\bl}=-\sum_{r=1}^n B_{ij}^{\br}\Gamma_{\br\bk}^{l}\,,\quad R_{ijk\bl}=0\,.
%\end{equation}
%Moreover the Ricci tensor of $g$ is of type $(1,1)$ and it can be described in terms
%of brackets as
%\begin{equation}
%{\rm Ric}_{i}^{j}=\sum_{q,r=1}^{n}\left(\frac{1}{4} B_{qr}^{\bi} B_{\bq\br}^{j} -
%\frac{1}{2} B_{iq}^{\br}B_{\bj\bq}^{r} \right)\,.
%\end{equation}
%\end{proposition}

\subsection{Gray conditions} In \cite{Gray} Gray considered some special classes of almost Hermitian manifolds characterized by some
identities involving the curvature tensor.
\begin{definition}\emph{
Let $(M,g,J)$ be an almost Hermitian manifold and let $R$ be
the curvature tensor of $g$. Then $R$ is said to satisfy}
\begin{itemize}
\item \emph{the \emph{first Gray identity} $({\rm G}_1)$ if
$R(Z_1,Z_2,\cdot,\cdot)=0$};

\vspace{0.1cm} \item \emph{the \emph{second Gray identity} $({\rm G}_2)$
if $R(Z_1,Z_2,Z_3,Z_4)=R(\overline{Z}_1,Z_2,Z_3,Z_4)=0$};

\vspace{0.1cm} \item \emph{the \emph{third Gray identity} $({\rm G}_3)$
if $R(\overline{Z}_1,Z_2,Z_3,Z_4)=0$};
\end{itemize}
\emph{for every $Z_1,Z_2,Z_3,Z_4\in\Gamma(T^{1,0}M)$. }
\end{definition}
Such identities are usually called  \emph{Gray identities}
and play a central role in the decomposition of the curvature of an almost Hermitian metric (see \cite{tri}).
In the present paper the second Gray identity plays a role in Theorem \ref{th1}.

We finish this section recalling Gray identities in their real version:
\begin{eqnarray*}
({\rm G} 1) && R(JX,JY,Z,T) = R(X,Y,Z,T)\,;\\
\vspace{0.1cm}
({\rm G} 2)&& R(X,Y,Z,W)-R(JX,JY,Z,W)=R(JX,Y, JZ,W)+R(JX,Y, Z,JW)\,;\\
\vspace{0.1cm} ({\rm G} 3) && R(JX,JY, JZ,JT) =R(X,Y,Z,T)\,.
\end{eqnarray*}

\section{Proofs of Theorem \ref{th1} and Theorem \ref{th2}}
Let $(G,J)$ be a complex Lie group, $\Gamma\subseteq G$ be a lattice, $M=G/\Gamma$ and let $g$ be a left-invariant Hermitian
metric on $M$. We denote by $(\mathfrak{g},\lb)$ the Lie algebra
of $G$, by $\mathfrak{z}$ the center of $\mathfrak{g}$ and by $\mathfrak{z}^\perp$ the orthogonal
complement of $\mathfrak{z}$. As stated in the introduction, we denote by $J^-$ the almost complex structure on $M$
induced by the relation
\begin{equation}
J^{-}(X)=\begin{cases}
+J(X)\quad \mbox{if }X\in \mathfrak{z} \\
-J(X)\quad \mbox{if }X\in \mathfrak{z}^\perp\,.
\end{cases}
\end{equation}

\begin{proof}[Proof of Theorem $\ref{th1}$]
Let $\{Z_1,\dots,Z_n\}$ be a left-invariant $(1,0)$-frame with respect to $J$.
Let $D^g$ be the Levi-Civita connection of $g$ and let $R^g$ be its curvature tensor. By using Koszul formula,
$[\mathfrak{g}^{1,0},\mathfrak{g}^{0,1}] = 0$ and $[\mathfrak{g}^{1,0},\mathfrak{g}^{1,0}] \subset \mathfrak{g}^{1,0} $ we get
\[
D^g_{Z_i} Z_j = \frac12\,[Z_i,Z_j] \, \,.
\]
Then
$$
\begin{aligned}
R^g(Z_i,Z_j)Z_k&= D^g_{Z_i}D^g_{Z_j}Z_k-D^g_{Z_j}D^g_{Z_i}Z_k-D^g_{[Z_i,Z_j]}Z_k\\
             &= \frac14\,\Big\{[Z_i,[Z_j,Z_k]] - [Z_j,[Z_i,Z_k]] - 2[[Z_i,Z_j],Z_k]\Big\}\\
             &= \frac14\,\Big\{[Z_i,[Z_j,Z_k]] + [Z_j,[Z_k,Z_i]] + 2[Z_k,[Z_i,Z_j]]\Big\}=\frac14 [Z_k,[Z_i,Z_j]]\,.
\end{aligned}
$$

Notice that the above formula for $R^g$ in terms of Lie brackets of left-invariant vector fields also holds for any other left-invariant Hermitian metric $g'$ on $G/\Gamma$. Indeed,
\[
R^g(W_1,W_2)W_3 = R^{g'}(W_1,W_2)W_3 = \frac14\,[W_3,[W_1,W_2]] \,
\]
for all triple
of left-invariant $(1,0)$-vector fields
$W_1,W_2,W_3$. This shows that the second Gray identity holds for $R^g$ if and only if it holds for $R^{g'}$ and we get $1\iff 2$.
Since $[Z,\overline{W}]=0$ for all left-invariant vector fields of type $(1,0)$, the above formula for $R^g$ implies that $1\iff 4$.

\smallskip

Let us prove $3 \Longrightarrow 4$. Let $\{Z_1,\dots,Z_n\}$ be a left-invariant $(1,0)$-frame with respect to $J^-$. Since  $J^{-}$ is anti-bi-invariant, then $[Z_i,Z_{\br}]=0$.
Moreover, the assumption that $(g,J^-)$ is quasi-K\"ahler implies  $ [Z_i,Z_r] \in\mathfrak{g}^{0,1}_{J^-}$
for $1 \leq i,r \leq n$. Indeed,
$$
\begin{aligned}
g([Z_i,Z_r],Z_{\bk}) &=g(D_i Z_r,Z_{\bk})- g(D_r Z_i,Z_{\bk})=-g(Z_r,D_i Z_{\bk})+ g(Z_i,D_r Z_{\bk})\\
                     &= -g(Z_r,D_{\bk} Z_{i})+ g(Z_i,D_{\bk} Z_{r})=0\,.
\end{aligned}
$$
This implies the following relations
$$
[\mathfrak{g}_{J^-}^{1,0},\mathfrak{g}_{J^-}^{1,0}] \subset \mathfrak{g}_{J^-}^{0,1}\,,\quad
[\mathfrak{g}_{J^-}^{0,1},\mathfrak{g}_{J^-}^{0,1}] \subset \mathfrak{g}_{J^-}^{1,0}\,,\quad
[\mathfrak{g}_{J^-}^{1,0},\mathfrak{g}_{J^-}^{0,1}]=0\,.
$$
that force $\mathfrak{g}$ to be $2$-step nilpotent.

\smallskip

Now we show $4 \Longrightarrow 3$. Since G is $2$-step nilpotent, then  $J_-$ is anti-bi-invariant.
That implies
$$
[\mathfrak{g}_{J^-}^{1,0},\mathfrak{g}_{J^-}^{1,0}] \subset \mathfrak{g}_{J^-}^{0,1}\,,\quad
[\mathfrak{g}_{J^-}^{0,1},\mathfrak{g}_{J^-}^{0,1}] \subset \mathfrak{g}_{J^-}^{1,0}\,,\quad
[\mathfrak{g}_{J^-}^{1,0},\mathfrak{g}_{J^-}^{0,1}]=0\,.
$$
Using Koszul's formula we obtain that $D_{\overline Z}W\in \mathfrak{g}_{J_-}^{1,0}$ for all $Z,W\in \mathfrak{g}_{J_-}^{1,0}$, which is equivalent to the quasi-K\"ahler condition.

%Let $g$ be any left-invariant almost Hermitian metric compatible with $J^-$ and let $D$ be its Levi-Civita connection. Since $G$ is $2$-step we have $\mathfrak{g} = \mathfrak{z} \oplus \mathfrak{z}^{\perp}$ and $[\mathfrak{z}^{\perp}, \mathfrak{z}^{\perp}] \subset \mathfrak{z}$, where $\mathfrak{z}$ is the center of $\mathfrak{g}$. By definition of $J^-$ we have that
%\[ \mathfrak{g}^{1,0}_{J^-} = (\mathfrak{z}^{\perp})^{1,0}_J \oplus (\mathfrak{z})^{0,1}_J \, \, .\]
%For $Z \in \mathfrak{g}^{1,0}_{J^-}$ let us denote $Z = Z_1 + Z_2$ the components according the above decomposition. Then for $Z,W \in \mathfrak{g}^{1,0}_{J^-}$ we get
%$$
%\begin{aligned}
%J^{-} [Z,W]&= J^{-}[Z_1,W_1]\\
%           &= -J[Z_1,W_1]  \mbox{ since } [Z_1,W_1] \in \mathfrak{z} \otimes \mathbb{C}\\
%           &= -i[Z_1,W_1] \mbox{ since } [Z_1,W_1] \in \mathfrak{z}^{1,0}_J \\
%           & = -i[Z,W] \, .
%\end{aligned}
%$$
%
%\noindent Hence we have
%\[
%[\mathfrak{g}^{1,0}_{J^-},\mathfrak{g}^{1,0}_{J^-}] \subset \mathfrak{g}^{0,1}_{J^-} \,\ .
%\]
%Then for $X,Y,Z \in \mathfrak{g}^{1,0}_{J^-}$, by using Koszul formula and the fact that $J^-$ is anti-bi-invariant we get
%\[ 2 g(D_{\overline{X}}Y,Z) = g(\overline{X},[Z,Y]) = 0 ,\]
%because $[Z,Y] \in \mathfrak{g}^{0,1}_{J^-} $. Then $D_{\overline{X}}Y \in \mathfrak{g}^{1,0}_{J^-}$ which means that the almost Hermitian structure $(g,J^-)$ is quasi-K\"ahler.

\smallskip

In order to show the last part of the statement we denote by $\nabla^g$ the Chern connection associated to $(g,J^-)$.
Let $Z = \{Z_1,\dots,Z_n\}$ be a $g$-unitary $(1,0)$-frame with respect to $J^-$ and
let $\nabla^Z$ be the flat connection associated to the frame $Z$. That means that $\nabla^Z$ satisfies $\nabla^Z Z_i=0$ for every
$i=1,\dots,n$. Then $\nabla^Z J^- = \nabla^Z g = 0$ and the $(1,1)$-part of the torsion of $\nabla^Z$ vanishes everywhere since $J^-$ is anti-bi-invariant. Then by the uniqueness of the Chern connection we get that $\nabla^Z = \nabla^g$, as required.
\end{proof}

\begin{remark}
We recall that if $(M=G/\Gamma,J)$ is a compact complex parallelisable manifold, then any left-invariant metric $g$ on $M$ is automatically
cosymplectic (see \cite{Abb}). Cosymplectic means that the Liouville $1$-form $d^*\omega$ associated to the fundamental form of $g$ vanishes everywhere on $M$.
\end{remark}
\subsection{Proof of Theorem \ref{th2}}
Let $(M,g)$ be a compact Riemannian manifold. As stated in the introduction, we consider the following spaces
\begin{itemize}
\item $\mathcal{Z}_g=\Big\{\mbox{almost complex structures on $M$ compatible to $g$}\Big\}\,;$

\vspace{0.1cm}
\item $\mathcal{Q}\mathcal{K}^{0}_{g}=\Big\{J\in\mathcal{Z}_g\,:\mbox{ $(g,J)$ is a quasi-K\"ahler Chern flat almost Hermitian structure} \Big\}$\,;

\vspace{0.1cm}
\item $\mathcal{CP}^{2}_{g}=\Big\{ J\in\mathcal{Z}_g\,:\mbox{ $(g,J)$ is a Chern-flat Hermitian structure and $R^g$ satisfies $({\rm G}_2)$} \Big\}$\,.
\end{itemize}

\begin{proof}[Proof of Theorem $\ref{th2}$]
Let $J\in\mathcal{Z}_g$ be a complex structure such that
$(g,J)$ is Chern flat. Then $J$ allows us to write $M=G/\Gamma$, for a Lie group $G$ and a lattice $\Gamma$ and $g$ is a left-invariant metric on $M$. If $J\in\mathcal{CP}^{2}_{g}$, then in view of
Theorem \ref{th1} $G$ is $2$-step nilpotent and  $J^-$ is in $\mathcal{Q}\mathcal{K}^{0}_{g}$. The map $J\mapsto J^-$ is clearly injective. In order to show that such a map is also surjective, let $J\in\mathcal{Q}\mathcal{K}^{0}_{g}$. Then in view of Theorem \ref{DVth1}, $J$ induces a structure of $2$-step nilmanifold on $M$ with respect to which $g$ and  $J$ are left-invariant. Hence by setting
$$
J^{+}(X)=\begin{cases}
+J(X)\quad \mbox{if }X\in \mathfrak{z} \\
-J(X)\quad \mbox{if }X\in \mathfrak{z}^\perp\,.
\end{cases}
$$
we obtain a complex parallelisable structure such that $J^+\mapsto J$.
\end{proof}

\begin{example}
Let $M= G/\Gamma$ be the Iwasawa manifold. We recall that this manifold is defined as
the quotient $G/\Gamma$, where $G$ is the $3$-dimensional complex Heisenberg group
$$
G=\left\{\left(
\begin{array}{ccc}
1           &z_1          & z_2 \\
0           &1            &z_3\\
0           &0            &1
\end{array}
\right)\,\,:\,\,z_1,z_2,z_3\in\mathbb{C}\right\}
$$
and $\Gamma$ is the subgroup of $G$ defined by restriction $z_i$ to be Gaussian
integers. Let $\mathfrak{g}$ be the Lie algebra of $G$. Then $\mathfrak{g}$ has a coframe
$\{e^1,\dots,e^6\}$ satisfying
$$
\begin{cases}
de^i=0\,,\quad i=1,2,3,4\,,\\
de^5=e^{13}+e^{42}\,,\\
de^6=e^{14}+e^{23}\,.
\end{cases}
$$
$M$ has the canonical Hermitian structure structure $(g,J_0)$, where $g=\sum_{i=1}^6
e^i\otimes e^i$ and $J_0$ is determined by the relations
$$
J_0e^1=-e^2\,,\quad J_0e^3=-e^4\,,\quad J_0e^5=-e^6\,.
$$
In terms of complex frame, $J_0$ has the following structure equations
$$
[Z_1,Z_2]=Z_{3}
$$
and the other brackets vanishes, $\{Z_1,Z_2,Z_3\}$ being the canonical
left-invariant unitary frame $e_i=\frac12 (e_i-\operatorname{i}J_0 e_i)$, $\{e^1,\dots,e^6\}$ being the dual frame associated to
$\{e_1,\dots, e_6\}$.
Hence
$(M,J_0)$ is a compact complex parallelisable $2$-step nilmanifold; moreover, since
$G$ is $2$-step nilpotent, the metric $g$ satisfies the second Gray identity with
respect to $J_0$. Since the center of $\mathfrak{g}$ is spanned by $e_5,e_6$, the
almost complex structure $J_{0}^{-}$ associated to $J_0$ has equations
$$
J_{0}^-e^1=-e^2\,,\quad J_{0}^{-}e^3=-e^4\,,\quad J_{0}^{-}e^5=e^6\,.
$$
This almost complex structure is quasi-K\"ahler Chern-flat and correspond to the
almost complex structure $J_3$ of \cite{AGS}.
\end{example}

\subsubsection{The {\rm SKT} condition}
In this subsection we consider compact SKT manifolds with trivial Chern holonomy. We
recall that an Hermitian manifold $(M,g,J)$ is called SKT (strong K\"ahler with
torsion) if its K\"ahler form $\omega$ satisfies
$$
\overline{\partial}\partial \omega=0\,.
$$
We have the following
\begin{proposition}
Let $(M,g,J)$ be a compact Chern-flat {\rm SKT} manifold. Then $(M,g,J)$ is a flat
K\"ahler torus.
\end{proposition}
\begin{proof}
Since the metric $g$ is by hypothesis Chern-flat and $J$ is integrable, $M$ is a
quotient of a Lie group by a lattice and $(g,J)$ is a left-invariant Hermitian
structure on $M$. Let $\{\zeta^1,\dots,\zeta^n\}$ be a left invariant unitary
coframe on $M$. Since $g$ is Chern-flat and $J$ is integrable we get
$$
\overline{\partial}\zeta_i=\partial\zeta_{\bi}=0\,.
$$
Moreover we can write $\omega=\frac{\operatorname{i}}{2}\sum_{i=1}^n\zeta^{i\bi}$.
Hence we can write
$$
\begin{aligned}
2\operatorname{i}\overline{\partial}\partial
\omega=&\sum_{i=1}^n\overline{\partial}\partial(\zeta^{i\bi})\sum_{i,r,s=1}^n\overline{\partial}(c_{rs}^{i}\zeta^{rs\bi})\sum_{i,r,s,m,l=1}^nc_{rs}^{i}c_{\bm\bl}^{\bi}\zeta^{rs\bm\bl}\,.
\end{aligned}
$$
Hence  the SKT equation $\overline{\partial}\partial \omega=0$ implies
$c_{rs}^{i}=0$, which forces $(M,g,J)$ to be a flat K\"ahler torus.
\end{proof}

\section{Anti-bi-invariant structures}\label{alg}
%According to Theorem \ref{DVth1}, compact quasi-K\"ahler Chern-flat manifolds are
%all $2$-step nilmanifolds whose Lie algebras are endowed with an anti-bi-invariant
%structure and any almost Hermitian metric.
In this section, we consider the problem
of classifying $2$-step nilpotent Lie algebras ({\it $2$-step algebras} for short
from now on) admitting anti-bi-invariant structures. We first show that such Lie
algebras are in bijection with complex $2$-step Lie algebras.  The correspondence is
given by a simple conjugation process, which in particular interchange the
complexification and the anti-complexification of a given real $2$-step Lie algebra.
As an application, we give explicit examples of $2$-step Lie algebras admitting an
anti-bi-invariant structure which do not come from the anti-complexification
procedure.  We finally study the existence of canonical almost Hermitian metrics for
a given ant-bi-invariant structure, that is, canonical metrics on a given
quasi-K\"ahler Chern-flat manifold.

\subsection{Almost complex Lie algebras}\label{dec}
Let $G$ be a real $2n$-dimensional Lie group with Lie algebra $\ggo$. An invariant
{\it almost complex} structure on $G$ is defined by a linear map
$J:\ggo\longrightarrow\ggo$ satisfying $J^2=-I$.  The pair $(\ggo,J)$ is called an
{\it almost complex Lie algebra}.  Two almost complex Lie algebras $(\ggo_1,J_1)$
and $(\ggo_2,J_2)$ are said to be {\it isomorphic} if there exists a Lie algebra
isomorphism $f:\ggo_1\longrightarrow\ggo_2$ such that $J_2=fJ_1f^{-1}$.  In order to
study the space of all almost complex Lie algebras of a given dimension, we consider
in what follows an approach based on varying Lie brackets rather than the structures
$J$'s on a fixed Lie algebra.

Let $W$ be a real $2n$-dimensional vector space and let $J \in\End(W)$ be a complex
structure, i.e. $J^2 = -{\rm Id}$, which will be fixed from now on.  Let
$V=\Lambda^2(W^*)\otimes W$ be the vector space of all skew-symmetric bilinear forms
from $W\times W$ on $W$.  Then
$$
\lca=\{\lb\in V:\lb\;\mbox{satisfies the Jacobi identity}\},
$$
is an algebraic subset of $V$ as the Jacobi identity can be written as zeroes of
polynomial functions.  $\lca$ is often called the {\it variety of Lie algebras} (of
dimension $2n$).  We note that each point $\lb$ of the variety $\lca$ can be
identified with an almost complex manifold by left translating $J$ on the simply
connected Lie group with Lie algebra $(W,\lb)$.

There is a natural action of $\Gl_{2n}(\RR):=\Gl(W)$ on $V$ given by
\begin{equation}\label{action}
(g.\lb)(X,Y)=g[g^{-1}X,g^{-1}Y], \qquad X,Y\in W, \quad g\in\Gl_{2n}(\RR),\quad
\lb\in V.
\end{equation}

Thus $\lca$ is $\Gl_{2n}(\RR)$-invariant and Lie algebra isomorphism classes are
precisely $\Gl_{2n}(\RR)$-orbits.  Let $\Gl_n(\CC)$ denote the subgroup of
$\Gl_{2n}(\RR)$ of those elements commuting with $J$.  We note that from this point
of view, two almost complex Lie algebras $\lb_1,\lb_2\in\lca$ are isomorphic if and
only if their $\Gl_n(\CC)$-orbits coincide, and a $\Gl_{2n}(\RR)$-orbit in $\lca$
parameterizes the set of all invariant almost complex structures on a given Lie
group.

We consider the following $\Gl_n(\CC)$-invariant subspaces of $V$:

\begin{itemize}
\item $V({\rm int}) := \{\lb \in V \, : \, [JX,JY]=[X,Y] + J[JX,Y] + J[X,JY] \}$,
\vspace{0.2cm}
\item $V({\rm ab}) := \{ \lb \in V \, : \, [JX,JY]=[X,Y] \}$,
\vspace{0.2cm}
\item $V(\mathbb{C}) := \{ \lb \in V \, : \, [JX,Y]=J[X,Y] \}$,
\vspace{0.2cm}
\item $V({\rm Ch}) := \{ \lb \in V \, : \, [JX,Y]=[X,JY] \}$,
\vspace{0.2cm}
\item $V(\overline{\mathbb{C}}) := \{ \lb \in V \, : \, [JX,Y]=-J[X,Y] \}$.
\end{itemize}

An almost complex Lie algebra $\lb\in\lca$ is therefore {\it integrable} if and only
if $\lb\in V({\rm int})$, and it is in addition {\it abelian} or {\it bi-invariant}
precisely when it belongs to $V({\rm ab})$ or $V(\mathbb{C})$, respectively.  It was
proved in \cite[Proposition 2.5]{DV} that the almost complex manifold determined by
$\lb\in\lca$ is Chern-flat if and only if $\lb\in V({\rm Ch})$, and in that case, it
is quasi-K\"ahler with respect to some (or equivalently to any) almost Hermitian
metric if and only if $\lb\in V(\overline{\mathbb{C}})$ (see \cite[Proposition
4.2]{DV}).  Notice that $V(\overline{\mathbb{C}}) \subset V({\rm Ch})$ and that the
set $V(\overline{\mathbb{C}})\cap\lca$ parameterizes the space of all
$2n$-dimensional anti-bi-invariant almost complex Lie algebras.

\begin{theorem}
The following decompositions into $\Gl_n(\mathbb{C})$-invariant subspaces hold:
\begin{itemize}
\item[(i)]  $V({\rm int}) = V({\rm ab}) \oplus V(\mathbb{C})$,
\vspace{0.1cm}
\item[(ii)]  $V(\rm{Ch}) =  V(\mathbb{C}) \oplus V(\overline{\mathbb{C}})$,
\vspace{0.1cm}
\item[(iii)]  $V = V(\rm{ab}) \oplus V(\mathbb{C}) \oplus V(\overline{\mathbb{C}})$.
\end{itemize}
\end{theorem}

\begin{proof} Let $\lb \in V({\rm int})$ be an integrable bracket. Then
\[ [X,Y] = \frac{[X,Y] + [JX,JY]}{2} + \frac{[X,Y] - [JX,JY]}{2} \, \]
where  $\frac{[X,Y] + [JX,JY]}{2} \in V(\rm{ab})$. Since $\lb \in V(\rm{int})$ the
bracket $\frac{[X,Y] - [JX,JY]}{2} \in V(\mathbb{C})$. Indeed, $[JX,Y] + [X,JY] J([X,Y] - [JX,JY]) $ is the integrability condition for $\lb$, and so part (i)
follows.

To prove (ii), we note that for any $\lb \in V({\rm Ch})$, \[ [X,Y] = \frac{[X,Y] -
J[JX,Y]}{2} + \frac{[X,Y] + J[JX,Y]}{2} \, , \] where $\frac{[X,Y] - J[JX,Y]}{2} \in
V(\mathbb{C})$ and $\frac{[X,Y] + J[JX,Y]}{2} \in V(\overline{\mathbb{C}})$ since
$\lb \in V({\rm Ch})$.

Finally, part (iii) follows from  \[ [X,Y] = \frac{[X,Y] + [JX,JY]}{2} + \frac{[X,Y]
- [JX,JY]}{2} \, \] where $\frac{[X,Y] + [JX,JY]}{2} \in V(\rm{int})$ and
$\frac{[X,Y] - [JX,JY]}{2} \in V(\rm{Ch})$, for any $\lb\in V$.
\end{proof}

\subsection{Correspondence with complex $2$-step algebras}\label{varphi}
Let $(W,J)$ be an $n$-dimensional complex vector space as above and assume that
$W=W_1\oplus W_2$ is a $J$-invariant direct sum decomposition, with $\dim{W_1}=2p$,
$\dim{W_2}=2q$ (in particular, $p+q=n$).  We consider the subspace $V_{pq}\subset V$
of those elements $\lb$ which satisfies $[W_1,W_1]\subset W_2$ and $[W,W_2]=0$.
Notice that $V_{pq}\subset\lca$ as any $\lb\in V_{pq}$ is a $2$-step algebra and so
the Jacobi identity is automatically satisfied. We can analogously define the
subspaces $V_{pq}({\rm int})$, $V_{pq}({\rm ab})$, $V_{pq}(\CC)$, $V_{pq}({\rm Ch})$
and $V_{pq}(\overline{\CC})$, with the advantage in this case that each point in any
of these subspaces actually corresponds to an almost complex structure on a
nilmanifold of the type suggested by its name.

Let us fix decompositions $W_1=W_1^{\RR}\oplus W_1^{i\RR}$, $W_2=W_2^{\RR}\oplus
W_2^{i\RR}$ with respect to which $J$ has the form
$$
J|_{W_1}=\left[
    \begin{smallmatrix}
      0 & -I \\
      I & 0 \\
    \end{smallmatrix}
  \right],  \qquad
J|_{W_2}=\left[
    \begin{smallmatrix}
      0 & -I \\
      I & 0 \\
    \end{smallmatrix}
  \right],
$$
and define $\vp\in\Gl_{2n}(\RR)$ as the operator whose matrix with respect to the
decomposition $W=W_1^{\RR}\oplus W_1^{i\RR}\oplus W_2^{\RR}\oplus W_2^{i\RR}$ is
given by
$$
\vp=\left[
    \begin{smallmatrix}
      I &  &  &  \\
       & -I &  &  \\
       &  & I &  \\
       &  &  & I \\
    \end{smallmatrix}
  \right].
$$
For any $\lb\in V_{pq}$ we define its {\it conjugate} $\lb^{-}\in V_{pq}$ by
$$
\lb^-:= [\vp\cdot,\vp\cdot] = \vp.\lb.
$$
In this way,
$$
\lb^-|_{W_1^{\RR}\times W_1^{\RR}}=\lb, \qquad \lb^-|_{W_1^{\RR}\times
W_1^{i\RR}}=-\lb, \qquad \lb^-|_{W_1^{i\RR}\times W_1^{i\RR}}=\lb.
$$
By using that $\vp|_{W_1}$ anti-commutes with $J|_{W_1}$, it is easy to see that
$$
\lb^-\in V_{pq}(\overline{\CC}) \quad\mbox{if and only if}\quad \lb\in V_{pq}(\CC).
$$
Thus the map $\lb\mapsto\lb^-$ determines an $\RR$-linear isomorphism between
$V_{pq}(\CC)$ and $V_{pq}(\overline{\CC})$ whose inverse is given by itself.
Moreover, since $\vp|_{W_2}=I$ we have that this isomorphism satisfies
$$
(g.\lb)^-=(\vp g\vp).\lb^-, \qquad\forall g\in\Gl_p(\CC)\times\Gl_q(\CC),
$$
from which follows that it takes $\Gl_p(\CC)\times\Gl_q(\CC)$-orbits in
$V_{pq}(\CC)$ into $\Gl_p(\CC)\times\Gl_q(\CC)$-orbits in $V_{pq}(\overline{\CC})$
and thus it also defines a homeomorphism
\begin{equation}\label{homeo}
V_{pq}(\CC)/\Gl_p(\CC)\times\Gl_q(\CC) \longrightarrow
V_{pq}(\overline{\CC})/\Gl_p(\CC)\times\Gl_q(\CC).
\end{equation}

The space of $2n$-dimensional $2$-step algebras is parameterized by the
$\Gl_{2n}(\RR)$-invariant algebraic subset
$$
\nca_2=\{\lb\in V:[W,[W,W]]=0\}\subset\lca.
$$
If for each $0\leq p,q\leq n$ such that $p+q=n$ we fix a $J$-invariant decomposition
and define $V_{pq}$ as above, then it is easy to see that
$$
\nca_2(\CC)=\bigcup_{p+q=n}\Gl_n(\CC).V_{pq}(\CC), \qquad
\nca_2(\overline{\CC})=\bigcup_{p+q=n}\Gl_n(\CC).V_{pq}(\overline{\CC}).
$$
We note that $\nca_2(\CC)$ is precisely the space of all $n$-dimensional complex
$2$-step algebras, if we use $J$ to define multiplication by $\im$ in $W$, and
$\nca_2(\overline{\CC})$ is the space of all anti-bi-invariant almost complex
$2$-step Lie algebras.

\begin{remark}
The set $\bigcup\limits_{p+q=n}V_{pq}$ does not parameterize the space of all almost
complex $2$-step algebras, as we have fixed on each $V_{pq}$ a structure $J$ which
satisfies certain extra properties relative to the Lie algebra structure.
\end{remark}

Concerning the quotients, we have the following disjoint unions:
$$
\nca_2(\CC)/\Gl_n(\CC)=\bigcup\limits_{p+q=n}V_{pq}^0(\CC)/\Gl_p(\CC)\times\Gl_q(\CC),
$$

$$
\nca_2(\overline{\CC})/\Gl_n(\CC)=\bigcup_{p+q=n}V_{pq}^0(\overline{\CC})/\Gl_p(\CC)
\times\Gl_q(\CC),
$$
where $V_{pq}^0$ is the open and dense subset of $V_{pq}$ defined by
$$
V_{pq}^0=\{\lb\in V_{pq}:[W_1,W_1]=W_2\}.
$$
The study of the set $\nca_2(\overline{\CC})/\Gl_n(\CC)$ parameterizing
anti-bi-invariant almost complex $2$-step Lie algebras up to isomorphism can
therefore be carried out by considering each
$$
V_{pq}^0(\overline{\CC})/\Gl_p(\CC)\times\Gl_q(\CC)
$$
separately, which by (\ref{homeo}) is in its turn homeomorphic to the space
$$
V_{pq}^0(\CC)/\Gl_p(\CC)\times\Gl_q(\CC)
$$
of all complex $2$-step algebras of up to isomorphism {\it type} $(p,q)$ (i.e.
$(p+q)$-dimensional and with the derived algebra of dimension $q$).

Thus the classification of anti-bi-invariant almost complex $2$-step Lie algebras,
which is parameterized by $\nca_2(\overline{\CC})/\Gl_n(\CC)$, is completely
equivalent to the classification of $n$-dimensional complex $2$-step algebras up to
isomorphism, that is, $\nca_2(\CC)/\Gl_n(\CC)$.  This is considered a wild problem
in many places in the literature. A complete classification is however known in the
following cases:

\begin{itemize}
\item Types $(p,1)$ for any $p\geq 2$: $\hg_{2r+1}\oplus\CC^{p-2r}$, $0\leq 2r\leq p$, where
$\hg_{2r+1}$ denotes the complex Heisenberg Lie algebra of dimension $2r+1$ (it
follows directly from the classification of $\CC$-valued forms). \vspace{0.1cm}
\item $n=p+q\leq 9$ (see \cite{GltTms}).
\vspace{0.1cm}
\item Types $(p,2)$ for any $p\geq 3$ (see \cite{Ggr}).
\vspace{0.1cm}
\item Type $(p,q)=(5,5)$ (see \cite{GltTms}).
\end{itemize}

For any $n\leq 8$ there are only finitely many $n$-dimensional complex $2$-step Lie
algebras, and $(6,3)$ is the only type in dimension $9$ which contains continuous
families ($2$-parameter).  There are also continuous families of type $(5,5)$ and
$(8,2)$.

\subsection{Complexifications and real forms}\label{comp}
In \cite[Section 4.1]{DV}, in order to get explicit examples of anti-bi-invariant
almost complex $2$-step Lie algebras, the following construction was given.
Starting from any real $2$-step Lie algebra $(\hg,\lb_{\hg})$, consider on
$\hg_{\CC}:=\hg\otimes\CC$ the almost complex structure defined by multiplication by
$\im$ and the Lie bracket given by
$$
[X\otimes a,Y\otimes b]:=\overline{ab}[X,Y]_{\hg}.
$$
It is easy to see that $\lb$ satisfies the Jacobi condition if and only if $\hg$ is
$2$-step nilpotent.  It is proved in \cite[Proposition 4.4]{DV} that this
construction always provides an anti-complex $2$-step algebra, which will be called
from now on the {\it anti-complexification} of $\hg$.  In that case, $\hg$ is said
to be a {\it real form} of $(\hg\otimes\CC,\lb,\im)$.

We can identify any real $2$-step algebra of type $(p,q)$ with a skew-symmetric
bilinear form $\lb^0:W_1^{\RR}\times W_1^{\RR}\longrightarrow W_2^{\RR}$ (we are
using the notation introduced in the above subsections).  The complexification and
the anti-complexification of $\lb^0$ are therefore given by
$$
\begin{array}{lccl}
\lb^0_{\CC}:W_1\times W_1\longrightarrow W_2, &\qquad&& \lb^0_{\overline{\CC}}:W_1\times W_1\longrightarrow W_2, \\ \\
\left\{\begin{array}{l}
\lb^0_{\CC}|_{W_1^{\RR}\times W_1^{\RR}}=\lb^0, \\ \\
\lb^0_{\CC}|_{W_1^{\RR}\times W_1^{i\RR}}=-J[\cdot,J\cdot]^0, \\ \\
\lb^0_{\CC}|_{W_1^{i\RR}\times W_1^{i\RR}}=[J\cdot,J\cdot]^0,
\end{array}\right.
&&& \left\{\begin{array}{l}
\lb^0_{\overline{\CC}}|_{W_1^{\RR}\times W_1^{\RR}}=\lb^0, \\ \\
\lb^0_{\overline{\CC}}|_{W_1^{\RR}\times W_1^{i\RR}}=J[\cdot,J\cdot]^0, \\ \\
\lb^0_{\overline{\CC}}|_{W_1^{i\RR}\times W_1^{i\RR}}=[J\cdot,J\cdot]^0,
\end{array}\right.
\end{array}
$$
It follows that their conjugates satisfy
\begin{equation}\label{ac}
(\lb^0_{\CC})^-=\lb^0_{\overline{\CC}},\qquad
(\lb^0_{\overline{\CC}})^-=\lb^0_{\CC},
\end{equation}
and thus the following properties can all be deduced from (\ref{ac}) and the
corresponding well known facts for complex $2$-step algebras:

\begin{itemize}
\item $\lb\in V_{pq}(\overline{\CC})$ is the anti-complexification of a real $2$-step
nilpotent algebra if and only if $\lb^-\in V_{pq}(\CC)$ is the complexification of
the same algebra.
\item[ ]
\item For any (complex) dimension $n=p+q\leq 8$, every $\lb\in V_{pq}(\overline{\CC})$
is the anti-complexification of a real $2$-step nilpotent algebra (see
\cite{Sly,Ggr}).
\item[ ]
\item The type $(p,q)=(6,3)$ is the only one with $n=p+q=9$ for which there exists
$\lb\in V_{pq}(\overline{\CC})$ such that $\lb$ is not the anti-complexification of
any real $2$-step nilpotent algebra (see \cite{GltTms}).  There are actually curves
of examples of this kind for $(p,q)=(6,3)$  (see Example \ref{ex63}).
\item[ ]
\item Given $\lb\in V_{pq}(\overline{\CC})$, there are at most a finite number up
to isomorphism of real $2$-step algebras whose anti-complexification is $\lb$, as it
is well known that any complex Lie algebra admits at most finitely many real forms
up to isomorphism (see Remark \ref{MJ}).
\end{itemize}

\begin{remark}\label{MJ}
Any complex $n$-dimensional Lie algebra can be identified with a point in $\Lambda^2(\CC^n)^*\otimes\CC^n$, and thus the set of its real forms up to isomorphism is parameterized by the intersection of its $\Gl_n(\CC)$-orbit with $\Lambda^2(\RR^n)^*\otimes\RR^n$ modulo the action by $\Gl_n(\RR)$.  It follows from \cite[Proposition 2.3]{BrlHrs} that such intersection contains at most finitely many $\Gl_n(\RR)$-orbits, as we are dealing with a rational representation of a reductive algebraic group.
\end{remark}

Let $K$ be a field of characteristic zero, and let $P_{q,p/2}$ denote the vector
space of all homogeneous polynomials of degree $p/2$ in $q$ variables with
coefficients in $K$. One can associate to any $2$-step algebra over $K$ of type
$(p,q)$, $p$ even, an element $f\in P_{q,p/2}$ called the {\it Pfaffian form} (see
\cite[Section 2]{ratform}). This provides an invariant in the following sense: if
two algebras are isomorphic then their respective Pfaffian forms $f$ and $g$ are
{\it projectively equivalent}, i.e. there exist $A\in\Gl_q(K)$ and $c\in K^*$ such
that
$$
f(x_1,...,x_q)=cg(A(x_1,...,x_q)).
$$
We can use this invariant to exhibit explicit examples of complex $2$-step algebras
without a real form as follows.  Assume that a complex $2$-step algebra $\ngo$ of
type $(p,q)$ has a Pfaffian form $f$ such that $S(f)/T(f)\notin\RR$ for two
$\Sl_q(\CC)$-invariant homogeneous polynomials
$S,T:P_{q,p/2}(\CC)\longrightarrow\CC$ of the same degree.  Then $\ngo$ does not
admit any real form.  Indeed, if $\ngo$ was the complexification of a real $2$-step
algebra, then $\ngo$ would be isomorphic as a complex Lie algebra to a Lie algebra
with real structural constants with respect to some basis and so its Pfaffain form
$g$ would have real coefficients (see \cite[Section 2]{ratform}).  The polynomials
$f$ and $g$ would therefore be projectively equivalent (recall that $A$ can be
assumed to belong to $\Sl_q(\CC)$ when $K=\CC$) and thus
$S(f)/T(f)=S(g)/T(g)\in\RR$, which is a contradiction.

We now use this criterium in explicit examples.

\begin{example}
Let us consider the complex $2$-step Lie algebra $\lb_t$, $t\in\CC$, of type $(8,2)$
defined in terms of the bases $\{ X_1,...,X_8\}$ of $W_1$ and $\{ Z_1,Z_2\}$ of
$W_2$ by
$$
\begin{array}{lll}
[X_1,X_5]_t=Z_1, & [X_2,X_6]_t=Z_1, & [X_3,X_7]_t=Z_1, \\

[X_4,X_8]_t=Z_1, & [X_2,X_5]_t=Z_2, & [X_3,X_6]_t=Z_2, \\

[X_4,X_7]_t=Z_2, & [X_1,X_8]=-Z_2 & [X_2,X_7]_t=-tZ_2.
\end{array}
$$

It is easy to check that its Pfaffian form is given by
$$
f_t(x,y)=x^4+tx^2y^2+y^4.
$$
The invariants $S$ and $T$ given in \cite[pp.150]{Dlg} of degree $2$ and $3$,
respectively, satisfy $S(f_t)=1+3t^2$ and $T(f_t)=t-t^3$, and thus for any $t\in\CC$
such that
$$
S^3(f_t)/T^2(f_t)=\frac{(1+3t^2)^3}{(t-t^3)^2}\notin\RR,
$$

the algebra $\lb_t$ does not admit any real form.
\end{example}

\begin{example}\label{ex63}
Let us consider the complex $2$-step Lie algebra $\lb_t$, $t\in\CC$, of type $(6,3)$
defined in terms of the bases $\{ X_1,...,X_6\}$ of $W_1$ and $\{ Z_1,...,Z_3\}$ of
$W_2$ by
$$
\begin{array}{lll}

[X_1,X_2]_t=tZ_1, & [X_3,X_4]_t=tZ_2, & [X_5,X_6]_t=tZ_3, \\

[X_5,X_4]_t=Z_1, & [X_1,X_6]_t=Z_2, & [X_3,X_2]_t=Z_3, \\

[X_3,X_6]_t=Z_1, & [X_5,X_2]_t=Z_2, & [X_1,X_4]_t=Z_3.
\end{array}
$$

By a straightforward computation one obtains that its Pfaffian form is
$$
f_t(x,y,z)=(t^3+2)xyz-t(x^3+y^3+z^3).
$$
The invariants $S$ and $T$ given in \cite[pp.160]{Dlg} of degree $4$ and $6$,
respectively, satisfy
$$
S(f_t)=-t^3\left(\tfrac{t^3+2}{6}\right)-\left(\tfrac{t^3+2}{6}\right)^4, \qquad
T(f_t)=t^6+20t^3\left(\tfrac{t^3+2}{6}\right)^3-8\left(\tfrac{t^3+2}{6}\right)^6,
$$
and thus for any $t\in\CC$ such that
$$
S^3(f_t)/T^2(f_t)=\tfrac{\left(-t^3\left(\tfrac{t^3+2}{6}\right)
-\left(\tfrac{t^3+2}{6}\right)^4\right)^3}{\left(t^6+20t^3
\left(\tfrac{t^3+2}{6}\right)^3-8\left(\tfrac{t^3+2}{6}\right)^6
\right)^2}\notin\RR,
$$
the algebra $\lb_t$ does not admit any real form.
\end{example}

In the above two examples, if $\lb_t\simeq\lb_s$ then
$S^3(f_t)/T^2(f_t)=S^3(f_s)/T^2(f_s)$, which can be used to get explicit examples of
curves of pairwise non-isomorphic complex $2$-step Lie algebras without a real form.
By the results described above, the conjugate of these curves provide examples of
anti-bi-invariant almost complex structures on $2$-step Lie algebras which are not
the anti-complexification of any real $2$-step Lie algebra.

\subsection{Canonical metrics for anti-bi-invariant structures}\label{can}
Let $N$ be a connected and simply connected nilpotent Lie group with Lie algebra
$\ngo$.  A left invariant Riemannian metric on $N$ is always identified with the
corresponding inner product $\ip$ on $\ngo$.  The Ricci operator of $(N,\ip)$ is
given by
\begin{equation}\label{ricci}
\begin{array}{rcl}
\la \Ricci_{\ip}X,Y\ra &=& -\unm\sum\la [X,X_i],X_j\ra
\la [Y,X_i],X_j\ra \\ \\
&& +\unc\sum\la [X_i,X_j],X\ra\la [X_i,X_j],Y\ra, \qquad\forall X,Y\in\ngo,
\end{array}
\end{equation}
where $\{ X_i\}$ is any orthonormal basis of $(\ngo,\ip)$.  It is easy to see that
$N$ can never admit an Einstein left invariant metric (i.e. $\Ricci_{\ip}=cI$),
unless it is abelian.  The following class of metrics play the role of most
distinguished or canonical metrics on nilmanifolds.

A left invariant metric on $N$ is said to be a {\it nilsoliton} if
$\Ricci_{\ip}=cI+D$ for some $c\in\RR$ and $D\in\Der(\ngo)$, where $\Der(\ngo)$ is
the vector space of all derivations of $\ngo$.   Nilsolitons are known to satisfy
the following properties:

\begin{itemize}
\item They are {\it Ricci solitons}, i.e. the solutions of the
Ricci flow $\dpar g(t)=-2\ricci_{g(t)}$ starting at them evolve only by scaling and
the action by diffeomorphisms (see \cite{libro}).

\item A given $N$ can admit at most one nilsoliton up to isometry and scaling among all its left-invariant metrics.

\item They are also characterized by minimizing the norm of the Ricci tensor among all
left invariant metrics on the given Lie group with the same scalar curvature.
Furthermore, they are the critical points of the functional square norm of the Ricci
tensor on the space of all nilmanifolds of a given dimension and scalar curvature.

\item Nilsolitons are precisely the nilpotent parts of Einstein solvmanifolds.
\end{itemize}

Nevertheless, the existence, structural and classification problems on nilsolitons
seem to be far from being satisfactory solved, if at all possible (see the recent
survey \cite{cruzchica} for further information on nilsolitons).

In the presence of an invariant almost complex structure on $N$, one may consider
the following definition for an almost Hermitian metric.  Given an almost complex
structure $J$ on $\ngo$, an almost Hermitian metric $\ip$ on $(\ngo,J)$ is called
{\it minimal} if $\Ricci_{\ip}^J=cI+D$ for some $c\in\RR$ and $D\in\Der(\ngo)$,
where
$$
\Ricci_{\ip}^J:=\unm\left(\Ricci_{\ip}-J\Ricci_{\ip}J\right)
$$
is the $(1,1)$-component of the Ricci operator $\Ricci_{\ip}$ of the almost
Hermitian manifold $(\ngo,J,\ip)$. The name is in some sense justified by the
following property: they minimize $||\Ricci_{\ip}^J||$ among all almost Hermitian
left invariant metrics on $(\ngo,J)$ with the same scalar curvature.  It is also
known that a minimal metric on $(\ngo,J)$ (if any) is unique up to isometry and
scaling, they develop soliton solutions for the `complex' Ricci flow and are
characterized as the critical points of a natural variational problem (see
\cite{minimal, praga} for further information on minimal compatible metrics for
geometric structures on nilmanifolds).

Let us now assume that $J$ is bi-invariant, and let $\ngo_{\CC}$ denote the complex
vector space $(\ngo,J)$.  Thus the Lie bracket $\lb$ of $\ngo$ defines a complex
nilpotent Lie algebra on $\ngo_{\CC}$.  It is easy to check by using formula
(\ref{ricci}) and the bi-invariance of $J$ that $[\Ricci_{\ip},J]=0$, and so
$\Ricci_{\ip}^J=\Ricci_{\ip}$; in particular, $\Ricci_{\ip}$ is a $\CC$-linear
operator on $\ngo_{\CC}$ and so is the derivation $D$ in case $\ip$ is minimal.

If we assume in addition that $\ngo_{\CC}$ is the complexification of a real
nilpotent Lie algebra $\ngo_{\RR}$, and if $\ip$ is given by
$$
\la X\otimes a,Y\otimes b\ra =\Re(a\overline{b})\la X,Y\ra_{\RR}, \qquad\forall
X,Y\in\ngo_{\RR},
$$
for some inner product $\ip_{\RR}$ on $\ngo_{\RR}$, then it is easy to see that with
respect to the orthogonal decomposition $\ngo=\ngo_{\RR}\oplus \im\ngo_{\RR}$ one
has:
$$
\Ricci_{\ip}=\left[
    \begin{smallmatrix}
      2\Ricci_{\ip_{\RR}} & 0 \\
      0 & 2\Ricci_{\ip_{\RR}} \\
    \end{smallmatrix}
  \right],
$$
where $\Ricci_{\ip_{\RR}}$ is the Ricci operator of $(\ngo_{\RR},\ip_{\RR})$.  By
using that $\ngo_{\RR}$ is a Lie subalgebra of $\ngo$ and the fact that
$\Ricci_{\ip}$ is $\CC$-linear, it is therefore easy to see that $(\ngo,\ip)$ is a
nilsoliton if and only if $(\ngo_{\RR},\ip_{\RR})$ is so.

On the other hand, when $\ngo$ is $2$-step nilpotent and $(\ngo,J)$ is bi-invariant,
we can consider the conjugate $\overline{\ngo}$ of $\ngo$ as defined in Subsection
\ref{varphi}.  Thus $(\overline{\ngo},J,\ip)$ is a quasi-K\"ahler Chern-flat
structure and analogously to the bi-invariant case, one can easily prove that the
$(1,1)$-component of its Ricci operator equals $\Ricci_{\ip}$.  This implies that
$(\overline{\ngo},\ip)$ is a nilsoliton if and only if $(\ngo,\ip)$ is so, and this
is also equivalent to the minimality of both the bi-invariant Hermitian structure
$(\ngo,J,\ip)$ and the almost Hermitian structure $(\overline{\ngo},J,\ip)$.

\begin{remark}
$(\overline{\ngo},J,\ip)$ is isometric to $(\ngo,J^{-},\ip)$, where $J^{-}:=\varphi
J\varphi^{-1}$ is obtained from $J$ by changing the sign on $W_1$.
\end{remark}

We summarize in the following proposition the results we have just obtained.

\begin{proposition}\label{equivs}
Let $(\ngo,J)$ be a real $2n$-dimensional nilpotent Lie algebra endowed with a
bi-invariant complex structure $J$.  For any Hermitian metric $\ip$, the following
conditions are equivalent:
\begin{itemize}
\item[(i)] $(\ngo,J,\ip)$ is minimal.

\item[(ii)] $(\ngo,\ip)$ is a nilsoliton.
\end{itemize}
If in addition the complex Lie algebra $\ngo_{\CC}=(\ngo,J)$ is the complexification
of an $n$-dimensional real Lie algebra $\ngo_{\RR}$, then a third equivalence holds:
\begin{itemize}
\item[(iii)] $(\ngo_{\RR},\ip_{\RR})$ is a nilsoliton, where $\ip_{\RR}:=\ip|_{\ngo_{\RR}\times\ngo_{\RR}}$.
\end{itemize}
Furthermore, assume that $\ngo$ is $2$-step nilpotent and consider $\overline{\ngo}$
the Lie algebra conjugate to $\ngo$.  Then parts {\rm (i)} and {\rm (ii)} (and also
{\rm (iii)} assuming that $\ngo_{\CC}=\ngo_{\RR}\otimes\CC$) are equivalent to the
following conditions:
\begin{itemize}
\item[(iv)] $(\overline{\ngo},J,\ip)$ is minimal.

\item[(v)] $(\overline{\ngo},\ip)$ is a nilsoliton.
\end{itemize}
\end{proposition}

We therefore obtain from known results on nilsolitons (see \cite[Section
6]{cruzchica}) that the anti-complexification of the following real $2$-step
algebras give rise to quasi-K\"ahler Chern-flat manifolds admitting a minimal almost
Hermitian metric:

\begin{itemize}
\item Any $2$-step algebra of dimension $\leq 8$ (only finitely many), and some continuous families for the types $(6,3)$, $(5,5)$ and $(8,3)$.

\vspace{0.1cm}
\item Any nilpotent Lie algebra with a codimension-one abelian ideal.

\vspace{0.1cm}
\item $H$-type Lie algebras (including real, complex and quaternionic Heisenberg algebras).

\vspace{0.1cm}
\item Nilradicals ($2$-step) of normal $j$-algebras (i.e. of
noncompact homogeneous K$\ddot{{\rm a}}$hler Einstein spaces), and of homogeneous
quaternionic K$\ddot{{\rm a}}$hler spaces.

\vspace{0.1cm}
\item Any $2$-step algebra of type $(p,q)$ with $p\leq 5$ and $(p,q)\ne (5,5), (5,6), (5,7)$.

\vspace{0.1cm}
\item The nilradical ($2$-step) of any parabolic subalgebra of any semisimple
Lie algebra.

\vspace{0.1cm}
\item Certain $2$-step algebras attached to (positive) graphs.

\vspace{0.1cm}
\item Nilpotent Lie algebras admitting a naturally reductive left invariant metric (which are all $2$-step and  constructed from representations of compact Lie groups).
\end{itemize}

Due to the `closeness' of $2$-step nilmanifolds to euclidean spaces in many aspects,
it was believed for a long time that any $2$-step algebra would admit a nilsoliton
metric.  This has recently been dismissed by the following curve $\ngo_{t},$  $t\in
(1,\infty)$, of pairwise non-isomorphic $2$-step algebras of type $(6,3)$, none of
which admit a nilsoliton metric (see \cite{Wll2}):
$$
\begin{array}{lll}
[X_5,X_4]_t=X_7, & [X_1,X_6]_t=X_8, &  [X_3,X_2]_t=X_9, \\

[X_3,X_6]_t=tX_7, & [X_5,X_2]_t=tX_8, &  [X_1,X_4]_t=tX_9, \\

[X_1,X_2]_t=X_7. &&
\end{array}
$$

Recall that hence the anti-complexifixation of these $2$-step algebras do not admit
minimal almost Hermitian metrics. Other examples of this kind are given by the
$2$-step algebras attached to non-positive graphs (in any dimension $\geq 11$), and
the real forms of six complex $2$-step algebras of type $(6,5)$ and three of type
$(7,5)$ (see \cite[Section 6]{cruzchica}).

\end{document}